\documentclass[11pt,a4paper]{amsart}
\usepackage{amsthm,amsmath,amssymb}
\usepackage[english]{babel}
\usepackage{graphicx}
\usepackage{amsmath,amssymb}
\newtheorem{theorem}{Theorem}[section]
\newtheorem{proposition}{Proposition}
\newtheorem{lemma}{Lemma}
\newtheorem{corollary}{Corollary}
\DeclareMathOperator{\mesh}{mesh}
\DeclareMathOperator{\stab}{Stab}
\theoremstyle{definition}
\newtheorem{definition}{Definition}

\theoremstyle{remark}

\numberwithin{equation}{section}

\begin{document}
\title{Decomposition complexity growth of finitely generated groups}
\author{Trevor Davila}
\begin{abstract}
Finite decomposition complexity and asymptotic dimension growth are two generalizations of M. Gromov's asymptotic dimension which can be used to prove property A for large classes of finitely generated groups of infinite asymptotic dimension. In this paper, we introduce the notion of decomposition complexity growth, which is a quasi-isometry invariant generalizing both finite decomposition complexity and dimension growth. We show that subexponential decomposition complexity growth implies property A, and is preserved by certain group and metric constructions.
\keywords{geometric group theory \and metric geometry \and asymptotic dimension \and finite decomposition complexity}
\end{abstract}

\maketitle

\section{Introduction}
\label{intro}
Asymptotic dimension was introduced by M. Gromov in \cite{GRO} to classify the large-scale geometry of finitely generated groups. In \cite{Y1} G. Yu proved the Novikov higher signature conjecture for groups with finite asymptotic dimension. In \cite{Y2} Yu introduced property A, a dimension-like property weaker than finite asymptotic dimension, and proved the coarse Baum-Connes conjecture for groups with property A. These results showed the dimension theory approach to coarse geometry to be quite fruitful. Groups with finite asymptotic dimension include finitely generated abelian groups, hyperbolic groups \cite{GRO}, mapping class groups \cite{BBF1}, Coxeter groups \cite{D2}, and groups acting properly and cocompactly on a finite dimensional $CAT(0)$ cube complex \cite{WR1}.

To study the dimension-like properties of metric spaces with infinite asymptotic dimension, several generalizations of asymptotic dimension which imply Yu's property A have been formulated. Finite decomposition complexity was introduced and was shown to imply property A in \cite{GTY1}. Groups with FDC include all groups of finite asymptotic dimension, countable subgroups of $GL(n,R)$ for any commutative ring $R$, and all elementary amenable groups \cite{GTY2}. In \cite{D1}, Dranishnikov introduced asymptotic dimension growth, showed that polynomial dimension growth implies property A, and this was strengthened to subexponential dimension growth by Ozawa. Groups with subexponential dimension growth include groups of finite asymptotic dimensino, wreath products $\mathbb{Z} \wr N$ with $N$ with $N$ virtually nilpotent \cite{D1}, iterated wreath products involving $\mathbb{Z}$, and coarse median groups \cite{ANWZ1}. Dranishnikov and Zarichnyi also introduced a weakening of finite decomposition complexity \cite{DZ1}, and showed that this still implies property A. However, some spaces and groups, most notably Thompson's group $F$, resist classification via these invariants. Also, the relationship between FDC and dimension growth is unclear.

In this paper, we introduce the notion of decomposition complexity growth, which generalizes both finite decomposition complexity and asymptotic dimension growth. We show that decomposition complexity growth is a quasi-isometry invariant. Our goal is to define the weakest possible version of decomposition complexity that still implies property A, which is subexponential decomposition growth. We show that finite decomposition complexity and subexponential asymptotic dimension growth both imply subexponential decomposition growth, that subexponential decomposition growth implies property A, and that decomposition complexity growth is preserved by some group and metric constructions. We also provide an example of a group whose finite decomposition complexity status and dimension growth are unknown, but which has subexponential decomposition growth.
\section{Preliminaries}
\label{prelim}
A quasi-isometric embedding of metric spaces is a map $f:X\to Y$ such that there exist constants $L,C > 0$ such that, for any $x,y \in X$
$$Ld(x,y)-C < d(f(x),f(y)) < Ld(x,y) + C.$$

A quasi-isometry is a quasi-isometric embedding $f:X\to Y$ such that 
there is some $C>0$ such that any $y \in Y$ has $x \in X$ such that $d(y,f(x)) < C$. Equivalently, a quasi-isometry is a quasi-isometric embedding $f:X \to Y$ such that there exists a quasi-isometric embedding $g:Y \to X$ and a constant $C$ such that $d(g \circ f (x), x) \leq C$ for all $x \in X$, and $d(f \circ g (y), y) \leq C$ for all $y \in Y$. Hence to show a property is a quasi-isometry invariant, it is enough to show it is pulled back by quasi-isometric embeddings.

Let $X$ be a metric space. For nonempty $A,B \subset X$, we let $d(A,B) = \inf \{ d(a,b) \, : \, a \in A, \, b \in B \}$.

Let $R > 0$. A family $\mathcal{U}$ of nonempty subsets of $X$ is $R$-disjoint if $d(A,B) > R$ for all $A,B \in \mathcal{U}$ with $A \neq B$.

A family $\mathcal{U}$ of subsets of $X$ is uniformly bounded if $mesh \, \mathcal{U} = \sup \{ diam(U) \, : \, U \in \mathcal{U} \} < \infty$.

\begin{definition}
A metric space $X$ \textit{$(R,n)$-decomposes} over a family of metric spaces $\mathcal{V}$ if there exists a family 
$$\mathcal{U}=\mathcal{U}_1 \cup \mathcal{U}_2 \dots \cup \mathcal{U}_n$$ of subsets of $X$ such that each $\mathcal{U}_i$ is $R$-disjoint, $\mathcal{U} \subset \mathcal{V}$, and
$$X = \bigcup \mathcal{U}$$
i.e. $\mathcal{U}$ is a cover of $X$. We write
$$X \stackrel{R,n}{\longrightarrow} \mathcal{V}.$$
\end{definition}

We recall the definition of asymptotic dimension growth. Originally defined in \cite{D1} in terms of multiplicities of covers with prescribed Lebesgue number, we instead use the definition from \cite{DS1} in terms of covers via disjoint families.

\begin{definition}[\cite{DS1}]
The \textit{asymptotic dimension growth function} $d_X : \mathbb{R}^+ \to \mathbb{N}$ of a metric space $X$ is defined so that $d_X(R)$ is the minimal $n$ such that there exists a uniformly bounded $\mathcal{U}$ with
$$X \stackrel{R,n}{\longrightarrow} \mathcal{U}.$$
\end{definition}

Finited decomposition complexity was introduced in \cite{GTY1}, but since we are interested in defining the weakest possible version of decomposition complexity that implies property A, we will use the weakening defined in \cite{DZ1}.
\begin{definition}
A metric space $X$ has \textit{straight finite decomposition complexity (sFDC)} if for any sequence of positive reals $R_1 \leq R_2 \leq R_3 \leq \dots$, there exist a positive integer $n$ and metric families $(\mathcal{V}_i)_{i=1}^n$ with a decomposition
$$X \stackrel{R_{1},2}{\longrightarrow} \mathcal{V}_{1} \stackrel{R_{2},2}{\longrightarrow} \mathcal{V}_2 \stackrel{R_{3},2}{\longrightarrow} \dots \stackrel{R_{n},2}{\longrightarrow} \mathcal{V}_n$$
with $\mathcal{V}_n$ uniformly bounded
\end{definition}

The following lemmas are standard in asymptotic dimension theory.
\begin{lemma}
Let $f:X\to Y$ be an $(L,C)$-quasi-isometric embedding, and let $\mathcal{V}$ be an $R$-disjoint family of subsets of $Y$. Then
$$f^{-1}(\mathcal{V}) = \{ f^{-1}(V) : V \in \mathcal{V} \}$$
is an $\frac{R-C}{L}$-disjoint family of subsets of $X$. Further, if $\mathcal{V}$ is uniformly bounded with $\mesh(\mathcal{V}) \leq D$, then $f^{-1}(\mathcal{V})$ is uniformly bounded with $mesh(f^{-1}(\mathcal{V})) \leq \frac{D + C}{L}$.
\end{lemma}
\begin{proof}
Let $f^{-1}(V_1),f^{-1}(V_2) \in f^{-1}(\mathcal{V})$ with $f^{-1}(V_1) \neq f^{-1}(V_2)$. Then $V_1 \neq V_2$, so $d(V_1, V_2) > R$. Hence for any $x\in f^{-1}(V_1), y \in f^{-1}(V_2)$ we have
$$R \leq d(f(x),f(y)) \leq Ld(x,y) + C,$$
so
$$\frac{R-C}{L} \leq d(x,y).$$
Hence $f^{-1}(\mathcal{V})$ is $\frac{R-C}{L}$-disjoint.

Further, if $\mesh(\mathcal{V}) \leq D$, then for any $V \in \mathcal{V}$ and $x,y \in f^{-1}(V)$ we have
$$Ld(x,y) - C \leq d(f(x),f(y)) \leq D.$$
Hence
$$d(x,y) \leq \frac{D+C}{L}.$$
Therefore $\mesh(f^{-1}(\mathcal{V})) \leq \frac{D+C}{L}$.

\end{proof}

\begin{lemma}
Let $f:X \to Y$ be an $(L,C)$-quasi-isometric embedding, and suppose we have a decomposition of metric families
$$\mathcal{U} \stackrel{R,n}{\longrightarrow} \mathcal{V}$$
where $\mathcal{U},\mathcal{V}$ are families of subsets of $Y$. Then $f^{-1}(\mathcal{U}),f^{-1}(\mathcal{V})$ are families of subsets of $X$ with a decomposition
$$f^{-1}(\mathcal{U}) \stackrel{R',n}{\longrightarrow} f^{-1}(\mathcal{V})$$
with $R' = \frac{R-C}{L}$.
\end{lemma}
\begin{proof}
We need to show that each $f^{-1}(U) \in f^{-1}(\mathcal{U})$ is a union of $n$ $R'-disjoint$ subfamilies of $f^{-1}(\mathcal{V})$. Let $U \in \mathcal{U}$. By assumption there are subfamilies $(\mathcal{V}_i)_{i=1}^n$ of $\mathcal{V}$ such that $\mathcal{V}_1 \cup \mathcal{V}_2 \cup \dots \cup \mathcal{V}_n$ is a cover of $U$. Then $f^{-1}(\mathcal{V}_1) \cup f^{-1}(\mathcal{V}_2) \cup \dots \cup f^{-1}(\mathcal{V}_n)$ is a cover of $f^{-1}(U)$, and by Lemma 1 each $f^{-1}(\mathcal{V}_i)$ is $\frac{R-C}{L}$-disjoint.
\end{proof}
We say a non-decreasing function $s:\mathbb{R}^+ \to \mathbb{N}$ is subexponential if 
$$\lim_{x \to \infty} \sqrt[x]{s(x)} = 1.$$
We say non-decreasing functions $s,t:\mathbb{R}^+ \to \mathbb{N}$ have the same \textit{growth} if there exist positive constants $a,c$ such that $s(ax) \geq t(x) - c$ and $t(ax) \geq s(x) - c$ for all $x > 0$, and we write $s \sim t$. We say $s$ has constant (polynomial, exponential) growth if it has the same growth as a constant (polynomial, exponential) function. Note that if $s$ is subexponential and $t$ has the same growth as $s$, then $t$ is subexponential.

Given a finitely generated group $G$ with finite symmetric (closed under inverses) generating set $S$, the \textit{word length metric} of $G$ is given by 
$$d(g,h) = |g^{-1}h|_S$$
where $| g |_S$ is the length of the shortest word in elements of $S$ equal to $g$ in $G$. Any word length metric is invariant under the action of $G$ on itself by left multiplication, and any two word length metrics on a given finitely generated group $G$ are quasi-isometric. Similarly, given a countable group $G$, there exists a proper (meaning locally finite), left-invariant metric on $G$, unique up to coarse equivalence. When we speak of a finitely generated group $G$ as a metric space, we assume it is equipped with a word length metric, and a countable group a proper left-invariant metric.

We say a metric space has \textit{bounded geometry} if it is locally finite, and for every $R \in \mathbb{R}^+$ there is $N(r)$ such that every $x \in X$ has $|B(x,r)| \leq N(r)$. Clearly all finitely generated and countable groups with metrics as above have bounded geometry.

\section{Decomposition Complexity Growth and Property A}
\label{decomp}
\begin{definition}
A non-decreasing function $s:\mathbb{R}^+\to \mathbb{N}$ is a \textit{decomposition complexity growth function} for a metric space $X$ if, for any sequence $R_1 \leq R_2 \leq R_3 \leq \dots$ of positive reals, there exist a positive integer $n$ and metric families $(\mathcal{V}_i)_{i=1}^n$ with a decomposition
$$X \stackrel{R_{1},s(R_1)}{\longrightarrow} \mathcal{V}_{1} \stackrel{R_{2},s(R_2)}{\longrightarrow} \mathcal{V}_2 \stackrel{R_{3},s(R_3)}{\longrightarrow} \dots \stackrel{R_{n},s(R_n)}{\longrightarrow} \mathcal{V}_n$$
with $\mathcal{V}_n$ uniformly bounded.
\end{definition}
\begin{definition}
We say a metric space $X$ has \textit{subexponential (constant, polynomial) decomposition growth} if there is a decomposition complexity growth function $s$ for $X$ such that $s$ has the same growth type as some subexponential (constant, polynomial, exponential) function.
\end{definition}
\begin{theorem}
If $s:\mathbb{R}^+ \to \mathbb{N}$ is a decomposition growth function for $Y$, and $f:X\to Y$ is a quasi-isometric embedding, then there is a function $t:\mathbb{R}^+ \to \mathbb{N}$ such that $s \sim t$ and $t$ is a decomposition growth function for $X$. Hence existence of a decomposition complexity growth function of a given growth type is a quasi-isometry invariant.
\end{theorem}
\begin{proof}
Suppose $f:X \to Y$ is an $(L,C)$-quasi-isometric embedding, and let $t(x) = s(Lx + C)$. Clearly $t$ and $s$ have the same growth. We want to show that $t$ is a decomposition complexity growth function for $X$. Let $R_1\leq R_2\leq \dots$ be a sequence of positive reals, and for each $i$ let $R_i' = LR_{i}+C$. Then there exist families of subsets of $Y$ $\mathcal{V}_i$ such that
$$Y \stackrel{R_1',s(R_1')}{\longrightarrow} \mathcal{V}_{1} \stackrel{R_2',s(R_2')}{\longrightarrow} \mathcal{V}_2 \stackrel{R_3',s(R_3')}{\longrightarrow} \dots \stackrel{R_n',s(R_n')}{\longrightarrow} \mathcal{V}_n$$
with $\mathcal{V}_n$ uniformly bounded. Now applying Lemma 2 to each decomposition above we obtain a decomposition
$$X \stackrel{R_1,s(R_1')}{\longrightarrow} f^{-1}(\mathcal{V}_{1}) \stackrel{R_2,s(R_2')}{\longrightarrow} f^{-1}(\mathcal{V}_2) \stackrel{R_3,s(R_3')}{\longrightarrow} \dots \stackrel{R_n,s(R_n')}{\longrightarrow} f^{-1}(\mathcal{V}_n).$$
But $s(R_i') = s(LR_i+C) = t(R_i)$ for each $i$, and $f^{-1}(\mathcal{V}_n)$ is uniformly bounded by Lemma 1. We conclude that $t$ is a decomposition complexity growth function for $X$.
\end{proof}

The following is immediate from Theorem 3.1.
\begin{corollary}
The property of subexponential (resp. constant, polynomial) decomposition growth is a quasi-isometry invariant.
\end{corollary}

Subexponential decomposition growth is a weakening of both straight finite decomposition complexity and subexponential asymptotic dimension growth.
\begin{proposition}
Any metric space with straight finite decomposition complexity has subexponential decomposition growth.
\end{proposition}
\begin{proof}
Suppose a metric space $X$ has straight finite decomposition complexity. Then for any sequence of positive reals $R_1 \leq R_2 \leq R_3 \leq \dots$ there exists a decomposition
$$X \stackrel{R_{1},2}{\longrightarrow} \mathcal{V}_{1} \stackrel{R_{2},2}{\longrightarrow} \mathcal{V}_2 \stackrel{R_{3},2}{\longrightarrow} \dots \stackrel{R_{n},2}{\longrightarrow} \mathcal{V}_n$$
such that $\mathcal{V}_n$ is uniformly bounded. Hence the constant function $s(x) = 2$ is a decomposition complexity growth function for $X$.
\end{proof}
\begin{proposition}
Any metric space with subexponential asymptotic dimension growth has subexponential decomposition growth.
\end{proposition}
\begin{proof}
Let $X$ be a metric space, and let $s$ be a subexponential function giving an upper bound on the dimension growth of $X$. Then for any positive real $R$ there is a decomposition
$$X \stackrel{R,s(R)}{\longrightarrow} \mathcal{V}$$
with $\mathcal{V}$ uniformly bounded.
\end{proof}
Instead of stating the original definition of G. Yu's property A, we take the equivalent characterization used in \cite{O1} and originally stated in \cite{W1}.
\begin{theorem}[\cite{W1}]
Let $X$ be a metric space with bounded geometry. $X$ has property $A$ if and only if there is a sequence of functions $f^n : X \to \ell_1(X)$ such that
\begin{enumerate}
\item $f_x^n \geq 0$ and  $\| f_x^n \| = 1$ for all $n$ and $x$,
\item for every $n$ there is $S_n > 0$ such that $\text{supp}f_x^n \subset B(x,S_n)$ for all $x \in X$,
\item for every $R > 0$, $\lim_n \sup \{ \| f_x^n-f_y^n \| : d(x,y) \leq R \} = 0.$
\end{enumerate}
\end{theorem}
We need the following lemma, extracted from the main proof in \cite{O1}.
\begin{lemma}[\cite{O1}]
Let $X$ be a metric space and $\mathcal{U}$ an open cover of $X$ with Lebesgue number $\geq \lambda \in \mathbb{N}$. Then there is a map $f: X \to \ell_1(\mathcal{U})$ such that any $x,y \in X$ with $d(x,y) = D$ with $2D+1 \leq \lambda$ have
$$\| f(x) - f(y) \| \leq 2(1 - m(\mathcal{U})^{-2D/\lambda})$$
\end{lemma}
\begin{proof}
Let $X, \mathcal{U}, \lambda, x, y, D$ be as above. For each $k \in \mathbb{N}$ let $S_x(k) = \{ U \in \mathcal{U} : B(x,k) \subset U \}$. For any finite set $S \subset \mathcal{U}$ let $\xi_S \in \ell_1(\mathcal{U})$ be defined for all $U \in S$ as 
$$\xi_S(U) = \frac{1}{|S|}$$
and $\xi_S(U) = 0$ if $U \notin S$. Note that, since $D = d(x,y)$, if $B(x,k+D) \subset U$, then $B(y,k) \subset U$, and also $B(x,k) \subset U$. Hence $S_x(k+D) \subset S_x(k) \cap S_y(k)$. Similarly $S_x(k) \cup S_y(k) \subset S_x(k-D)$ if $k \geq D$. Hence for any $k \geq D$
$$\| \xi_{S_x(k)} - \xi_{S_y(k)} \| = 2 \bigg ( 1 - \dfrac{| S_x(k) \cap S_y(k) |}{\text{max} \{ |S_x(k)|, |S_y(k)| \} } \bigg ) \leq 2 \bigg ( 1 - \dfrac{| S_x(k + D) |}{| S_x(k - D) |} \bigg ).$$
For each $z \in X$ let
$$f_z = \frac{1}{\lambda} \sum_{k=\lambda+1}^{2\lambda}\xi_{S_z(k)} \in \ell_1(\mathcal{U}).$$
Then
$$\| f_x - f_y \| \leq \frac{1}{\lambda} \sum_{k=\lambda+1}^{2\lambda} \| \xi_{S_x(k)} - \xi_{S_y(k)} \| \leq \frac{2}{\lambda}\sum_{k=\lambda+1}^{2\lambda} \bigg ( 1 - \dfrac{| S_x(k + D) |}{| S_x(k - D) |}\bigg ).$$
But
$$\frac{1}{\lambda} \sum_{k=\lambda+1}^{2\lambda} \dfrac{| S_x(k + D) |}{| S_x(k - D) |} \geq \bigg ( \prod_{k=\lambda+1}^{2\lambda} \dfrac{| S_x(k + D) |}{| S_x(k - D) |} \bigg )^{1/\lambda} = \bigg ( \dfrac{ \prod_{k=2\lambda - D +1}^{2\lambda + D}| S_x(k) |}{ \prod_{k=\lambda - D + 1}^{\lambda + D}| S_x(k) |} \bigg )^{1/\lambda} \geq m(\mathcal{U})^{-2D/\lambda}$$
Hence
$$\| f_x - f_y \| \leq 2(1 - m(\mathcal{U})^{-2D/\lambda} ).$$
\end{proof}
\begin{theorem}
Any bounded geometry metric space with subexponential decomposition growth has property $A$.
\end{theorem}
\begin{proof}
Let $X$ be a bounded geometry metric space with subexponential decomposition growth function $s : \mathbb{R^+} \to \mathbb{N}$. We construct functions $f_x^n$ as in Theorem 3.2. Fix $n$. For each $i \geq 1$ let $R_i \geq 2n + 1$ be such that \newline $2(1 - s(3R_i)^{-2n/R_i}) \leq \frac{1}{2^i n}$. Such $R_i$ exist since $s$ has subexponential growth. Now let
$$X \stackrel{3R_{1},s(3R_1)}{\longrightarrow} \mathcal{V}_{1} \stackrel{3R_{2},s(3R_2)}{\longrightarrow} \mathcal{V}_2 \stackrel{3R_{3},s(3R_3)}{\longrightarrow} \dots \stackrel{3R_{m},s(3R_m)}{\longrightarrow} \mathcal{V}_l$$
be a decomposition with $\mathcal{V}_l$ uniformly bounded. We need to thicken the sets in the decomposition into open sets. For each $V \in {V}_i$, $i \leq l - 1$, let $\mathcal{V}_V \subset \mathcal{V}_{i+1}$ be such that 
$$V \stackrel{3R_{i+1}, \, s(3R_{i+1})}{\longrightarrow} \mathcal{V}_{V}.$$
We may assume if $V_1 \neq V_2$ then $\mathcal{V}_{V_1} \cap \mathcal{V}_{V_2} = \emptyset$, by taking a disjoint union if necessary. Let
$$\mathcal{U}_1 = \{ B(V, R_1) : V \in \mathcal{V}_1 \}.$$
Having defined $\mathcal{U}_i$, $i \leq l - 1$, define $\mathcal{U}_U$ for each $U \in \mathcal{U}_i$ as
$$\mathcal{U}_U = \{ B(V', R_{i+1}) \cap U : V' \in \mathcal{V}_V \}$$
where $U$ is the thickening of $V$, i.e. $U = B(V, R_i)$. Now define
$$\mathcal{U}_{i+1} = \bigcup_{U \in \mathcal{U}_{i}} \mathcal{U}_U.$$
Then for each $0 \leq i \leq l - 1$ and $U \in \mathcal{U}_i$, $\mathcal{U}_U \subset \mathcal{U}_{i+1}$ is an open over of $U$ with Lebesgue number $\geq R_{i+1}$ and multiplicity $\leq s(3R_{i+1})$.
\newline
Let $g^1: X \to \ell_1( \mathcal{U}_1 )$ be the function given by Lemma 3.
Now suppose we have defined $g^i : X \to \ell_1(\mathcal{U}_i)$, and if $d(x,y) \leq n $ then
$$\|  g^i_x - g^i_y \| \leq \sum_{j = 1}^i 2(1 - s(R_{j})^{-4n/R_{j}}).$$
We want to construct $g^{i+1}: X \to \ell_1( \mathcal{U}_{i+1} )$. For each $U \in \mathcal{U}_i$, let
$$S_x(k,U) = \{ U' \in \mathcal{U}_U : B(x,k) \cap U \subset U' \}$$
For each $U \in \mathcal{U}_i$, apply lemma 1.5 to get a function $g^U : U \to \ell_1(\mathcal{U}_U)$ which for $d(x,y) \leq n$ has
$$\| g_x^U - g_y^U \| \leq 2(1 - s(R_{i+1})^{-2n/R_{i+1}}).$$
We now define $g^{i+1}_x \in \ell_1(\mathcal{U}_{i+1})$. If $U \in \mathcal{U}_i$ with $x \notin U$, let $g^{i+1}_x(U') = 0$ for every $U' \in \mathcal{U}_U$. Now assume $U \in \mathcal{U}_i$ with $x \in U$, and for each $U' \in \mathcal{U}_U$ define
$$g_x^{i+1}(U') = g_x^i(U) \cdot g_x^U(U')$$
It follows from $\| g^i_x \| = 1$ and $\| g_x^U \| = 1$ that $\| g_x^{i+1} \| = 1$. If $y \notin U$, then for any $U' \in \mathcal{U}_U$
$$| g^{i+1}_x(U') - g^{i+1}_y(U')| = |g^{i+1}_x(U')|$$
Now assume $x,y \in U$ and $U' \in \mathcal{U}_U$. Then
\begin{equation}
\begin{aligned}
| g^{i+1}_x(U') - g^{i+1}_y(U')| &= |g_x^i(U) \cdot g_x^U(U') - g_y^i(U) \cdot g_y^U(U')| \\
& \leq |g_x^i(U) \cdot g_x^U(U') - g_x^i(U) \cdot g_y^U(U')| + |g_x^i(U) \cdot g_y^U(U') - g_y^i(U) \cdot g_y^U(U')|\\
& = |g^i_x(U)||g_x^U(U') - g_y^U(U')| + |g_y^U(U')||g_x^i(U) - g_y^i(U)|\notag
\end{aligned}
\end{equation}
Hence if $d(x,y) \leq n$ then
\begin{equation}
\begin{aligned}
\sum_{U' \in \mathcal{U}_U} | g^{i+1}_x(U') - g^{i+1}_y(U')| & \leq |g^i_x(U)| \cdot \| g^U_x - g^U_y\| + |g^i_x(U) - g^i_y(U)| \\
& \leq |g^i_x(U)| \cdot 2(1 - s(R_{i+1})^{-4n/R_{i+1}}) + |g^i_x(U) - g^i_y(U)| \notag
\end{aligned}
\end{equation}
If $x,y \notin U$, then the above is obvious since each $ | g^{i+1}_x(U') - g^{i+1}_y(U')| = 0$. If $x \in U$ but $y \notin U$, then the above is also true:
\begin{equation}
\begin{aligned}
\sum_{U' \in \mathcal{U}_U} | g^{i+1}_x(U') - g^{i+1}_y(U')| &= \sum_{U' \in \mathcal{U}_U} | g^{i+1}_x(U')| \\
& = \sum_{U' \in \mathcal{U}_U}  | g_x^i(U) \cdot g_x^U(U'| \\
& = | g_x^i(U) | = |g_x^i(U) - g_y^i(U) | \\
& \leq |g^i_x(U)| \cdot 2(1 - s(R_{i+1})^{-2n/R_{i+1}}) + |g^i_x(U) - g^i_y(U)| \notag
\end{aligned}
\end{equation}
Now summing over all $U \in \mathcal{U}_i$ we obtain
\begin{equation}
\begin{aligned}
\| g^{i+1}_x - g^{i+1}_y \| &= \sum_{U \in \mathcal{U}_i} \sum_{U' \in \mathcal{U}_U} |g^{i+1}_x(U') - g^{i+1}_y(U')| \\
& \leq \sum_{U \in \mathcal{U}_i} |g^i_x(U)| \cdot 2(1 - s(R_{i+1})^{-2n/R_{i+1}}) + |g^i_x(U) - g^i_y(U)| \\
& = 2(1 - s(R_{i+1})^{-2n/R_{i+1}}) + \| g^i_x - g^i_y \| \\
& = \sum_{j = 1}^{i+1} 2(1 - s(R_{j})^{-2n/R_{j}}) \notag
\end{aligned}
\end{equation}
Finally we obtain $g^l : X \to \ell_1(\mathcal{U}_l)$. For each $U \in \mathcal{U}_l$ let $x_U \in U$. Then the map $U \mapsto x_U$ induces a map $\pi : \ell_1(\mathcal{U}_l) \to \ell_1(X)$ with $\| \pi( h_1 ) - \pi(h_2) \| \leq \| h_1 - h_2 \|$ for $h_1, h_2 \in \ell_1(\mathcal{U})$. Let $f^n : X \to \ell_1(X)$ be defined by $f^n_x = \pi(g^l_x)$. Then if $d(x,y) \leq n$ then
$$\| f^n_x - f^n_y \| \leq  \| g^l_x - g^l_y \| \leq \sum_{i = 1}^{l} 2(1 - s(R_{i})^{-2n/R_{i}}) \leq \sum_{i=1}^l \frac{1}{2^in} \leq 1/n$$
Hence the $f^n$ satisfy (3) of Theorem 3.2. $\mathcal{U}_l$ is uniformly bounded and for each $U \in \mathcal{U}_l$ we have $f_x^n(x_U) \neq 0$ if and only if $x \in U$. Hence $f_x^n = 0$ outside $B(x, \mesh(\mathcal{U}_l))$, so (2) of Theorem 3.2 is satisfied. (1) is satisfied since $\| f^n \| = \| g^l \| = 1$. We conclude that $X$ has property A.
\end{proof}

\section{Preservation by Metric and Group Constructions}
\label{pres}

\begin{theorem}
If $s$ is a decomposition complexity growth function for $X$, and $t$ is a decomposition complexity growth function for $Y$, then $s \cdot t$ is a decomposition complexity growth function for $X \times Y$.
\end{theorem}
\begin{proof}
We equip $X\times Y$ with the metric $d((x_1,y_1),(x_2,y_2)) = d_X(x_1, x_2) + d_Y(y_1, y_2).$ Let $R_1 \leq R_2 \leq R_3 \leq \dots$ be a sequence of positive reals, and let $(\mathcal{U}_i)_{i=1}^n, (\mathcal{V}_i)_{i=1}^m$ be such that $\mathcal{U}_n$ and $\mathcal{V}_m$ are uniformly bounded, and there are decompositions
$$X \stackrel{R_1,s(R_1)}{\longrightarrow} \mathcal{U}_{1} \stackrel{R_{2},s(R_2)}{\longrightarrow} \mathcal{U}_2 \stackrel{R_{3},s(R_3)}{\longrightarrow} \dots \stackrel{R_{n},s(R_n)}{\longrightarrow} \mathcal{U}_n$$
and
$$Y \stackrel{R_1,t(R_1)}{\longrightarrow} \mathcal{V}_{1} \stackrel{R_2,t(R_2)}{\longrightarrow} \mathcal{V}_2 \stackrel{R_3,t(R_3)}{\longrightarrow} \dots \stackrel{R_m,t(R_m)}{\longrightarrow} \mathcal{V}_m.$$
If $n<m$, then we can add $m-n$ trivial decomposition steps in the decomposition of $X$, where each $\mathcal{U}_{n+l} = \mathcal{U}_n$, and each $U \in \mathcal{U}_{n+l}$ decomposes into a single family $\{ U \}$. So we assume without loss of generality that $n=m$. For metric families $\mathcal{U}, \mathcal{V}$, define
$$\mathcal{U} \times \mathcal{V} = \{ U \times V : U \in \mathcal{U}, V \in \mathcal{V} \}.$$
Now we show that each $\mathcal{U}_i \times \mathcal{V}_i$ $(R_i, s(R_i) \cdot t(R_i))$-decomposes over $\mathcal{U}_{i+1} \times \mathcal{V}_{i+1}$. Let $i \leq n$, and let $U \in \mathcal{U}_i$ and $V \in \mathcal{V}_i$. Then there exist subfamilies of $\mathcal{U}_{i+1}$, say $(\mathcal{U}'_j)_{j=1}^{s(R_i)}$, and subfamilies of $\mathcal{V}_{i+1}$ $(\mathcal{V}'_k)_{k=1}^{t(R_i)}$, such that each $\mathcal{U}'_j$ is $R_i$-disjoint, each $\mathcal{V}'_k$ is $R_i$-disjoint, the $\mathcal{U}_j'$ form a cover of $X$, and the $\mathcal{V}_k'$ form a cover of $Y$. Then $(\mathcal{U}'_j \times \mathcal{V}'_k)_{j,k}$ form a cover of $X \times Y$ consisting of $s(R_i) \cdot t(R_i)$ families, and each $\mathcal{U}'_j \times \mathcal{V}'_k$ is $R_i$-disjoint. Hence we have a decomposition
$$X \stackrel{R_1,s(R_1)\cdot t(R_1)}{\longrightarrow} \mathcal{U}_{1} \times \mathcal{V}_1 \stackrel{R_{2},s(R_2)\cdot t(R_2)}{\longrightarrow} \mathcal{U}_2 \times \mathcal{V}_2 \stackrel{R_{3},s(R_3) \cdot t(R_3)}{\longrightarrow} \dots \stackrel{R_{n},s(R_n) \cdot t(R_n)}{\longrightarrow} \mathcal{U}_n \times \mathcal{V}_n$$
and clearly $\mathcal{U}_n \times \mathcal{V}_n$ is uniformly bounded, since $\mathcal{U}_n$ and $\mathcal{V}_n$ are. Hence $s \cdot t$ is a decomposition complexity growth function for $X \times Y$.
\end{proof}
\begin{corollary}
If $X$ and $Y$ both have subexponential (resp. constant, polynomial) decomposition growth, $X \times Y$ has subexponential (resp. constant, polynomial) decomposition growth.
\end{corollary}
Toward a group extension stability theorem for decomposition complexity growth, we follow the approach of the fibering theorem for finite asymptotic dimension in \cite{BD1}. Given an action of a group $G$ on a metric space $X$ with $x\in X$, let $\stab_R(x) = \{ g \in G : d(gx, x) \leq R \}$, called the $R$-stabilizer of $x$.
\begin{theorem}
Suppose $G$ is a finitely generated group acting transitively by isometries on a metric space $X$, $X$ has a decomposition complexity growth function $s\geq 2$, and for every $R > 0$ the $R$-stabilizers of the action of $G$ on $X$ have straight finite decomposition complexity. Then there is a function $t \sim s$ such that $t$ is a decomposition complexity growth function for $G$.
\end{theorem}
\begin{proof}
Fix a finite, symmetric generating set $S$ of $G$ and consider the word length metric of $G$ with respect to $S$, and fix a base point $x_0 \in X$. Let $\pi : G \to X$ be defined by $\pi(g) = gx_0$. If $L = \max{ \{d(x_0, \sigma x_0) : \sigma \in S \} }$,
then for any $g \in G$
\begin{equation}
\begin{aligned}
d(\pi(g), \pi(h)) &= d(gx_0, hx_0) \\
& = d(h^{-1}gx_0, x_0)\\
& \leq L \cdot |h^{-1}g|_S\\
& = Ld(g,h)\notag
\end{aligned}
\end{equation}
i.e. $\pi$ is $L$-Lipschitz.

Now Let $R_1 \leq R_2 \leq R_3 \leq \dots$ be a sequence of positive reals. Then we have a decomposition
$$X \stackrel{LR_{1},s(LR_1)}{\longrightarrow} \mathcal{V}_{1} \stackrel{LR_{2},s(LR_2)}{\longrightarrow} \mathcal{V}_2 \stackrel{LR_{3},s(LR_3)}{\longrightarrow} \dots \stackrel{LR_{n},s(LR_n)}{\longrightarrow} \mathcal{V}_n$$
with $\mathcal{V}_n$ uniformly bounded, and say $D = \mesh{\mathcal{V}_n}$. Then by Lemma 1 and 2, pulling the decomposition back by $\pi$ gives us a decomposition
$$G \stackrel{R_{1},s(LR_1)}{\longrightarrow} \pi^{-1}(\mathcal{V}_{1}) \stackrel{R_{2},s(LR_2)}{\longrightarrow} \pi^{-1}(\mathcal{V}_{2}) \stackrel{R_{3},s(LR_3)}{\longrightarrow} \dots \stackrel{R_{n},s(LR_n)}{\longrightarrow} \pi^{-1}(\mathcal{V}_{n})$$
but $\pi^{-1}(\mathcal{V}_n)$ is not uniformly bounded in general. So now we must decompose $\pi^{-1}(\mathcal{V}_n)$ into a uniformly bounded family, and we are done.

For each $U \in \pi^{-1}(\mathcal{V}_n)$, fix $g_U \in U$. Then for any $g \in U$, 
\begin{equation}
\begin{aligned}
d(g_U^{-1}gx_0, x_0) &= d(gx_0, g_Ux_0) \\
& = d(\pi(g), \pi(g_U))\\
& \leq D\notag
\end{aligned}
\end{equation}
since $\pi(U) \in \mathcal{V}_n$ has diameter $\leq D$. Hence $g_U^{-1}U \subset \stab_D(x_0)$. Therefore each $U \in \pi^{-1}(\mathcal{V}_n)$ is isometric to a subset of $\stab_D(x_0)$ via multiplication on the left by $g_U$. But $\stab_D(x_0)$ has $sFDC$ by assumption, so we have a decomposition
$$\stab_D(x_0) \stackrel{R_{n+1},2}{\longrightarrow} \mathcal{U}_{1} \stackrel{R_{n+2},2}{\longrightarrow} \mathcal{U}_{2} \stackrel{R_{n+3},2}{\longrightarrow} \dots \stackrel{R_{n+k},2}{\longrightarrow} \mathcal{U}_{k}$$
for some $k$, such that $\mathcal{U}_k$ is uniformly bounded, say with $\mesh(\mathcal{U}_k) = C$.

Now we can pull back each $\mathcal{U}_i$ by multiplication by $g_U^{-1}$ to get a decomposition of each $U \in \pi^{-1}(\mathcal{V}_n)$: for each $U \in \pi^{-1}(\mathcal{V}_n)$ and $1 \leq i \leq k$, we define
$$\mathcal{W}_i^U = \{ (g_U^{-1} U') \cap U : U' \in \mathcal{U}_i \}.$$
Then for each $U \in \pi^{-1}(\mathcal{V}_n)$ we have a decomposition
$$U \stackrel{R_{n+1},2}{\longrightarrow} \mathcal{W}_{1}^U \stackrel{R_{n+2},2}{\longrightarrow} \mathcal{W}_{2}^U \stackrel{R_{n+3},2}{\longrightarrow} \dots \stackrel{R_{n+k},2}{\longrightarrow} \mathcal{W}_{k}^U$$
with $\mathcal{W}_k^U$ uniformly bounded with $\mesh(\mathcal{W}_k^U) = C$. Hence if we define for each $1 \leq i \leq k$
$$\mathcal{W}_i = \bigcup_{U \in \pi^{-1}(\mathcal{V}_n)} \mathcal{W}_i^U$$
we have a decomposition
$$\pi^{-1}(\mathcal{V}_n) \stackrel{R_{n+1},2}{\longrightarrow} \mathcal{W}_{1} \stackrel{R_{n+2},2}{\longrightarrow} \mathcal{W}_{2} \stackrel{R_{n+3},2}{\longrightarrow} \dots \stackrel{R_{n+k},2}{\longrightarrow} \mathcal{W}_{k}$$
with $\mesh(\mathcal{W}_k) = C$. Append this to the decomposition of $G$ ending in $\pi^{-1}(\mathcal{V}_n)$ above, and we have shown that the function $t$ defined by $t(R) = s(LR)$ is a decomposition complexity growth function for $G$.
\end{proof}

\begin{theorem}
If
$1 \rightarrow K \rightarrow G \rightarrow H \rightarrow 1$
is a short exact sequence of groups such that $K$ is a countable group with straight finite decomposition complexity, $G$ is finitely generated, and $H$ has decomposition complexity growth function $s$, then $G$ has a decomposition complexity growth function $t \sim s$.
\end{theorem}
\begin{proof}
Let $G$ be equipped with the word length metric of a finite symmetric generating set $S$. If $\pi: G \to H$ is the quotient map, then $\pi(S)$ is a generating set for $H$, and $G$ acts transitively by isometries on $H$ via left multiplication, when $H$ is equipped with the word length metric of $\pi(S)$. Also, by the proof of Theorem 7 of \cite{BD1}, if $e \in H$ is the identity, then $\stab_R(e) = B_R(K) = \{ g \in G : d(g, K) \leq R \}$, and $B_R(K)$ is quasi-isometric to $K$ via contraction. Hence $\stab_R(e) = B_R(k)$ has straight finite decomposition complexity, since sFDC is a quasi-isometry invariant (\cite{DZ1}, Theorem 3.1). Therefore the action of $G$ on $H$ satisfies the conditions of Theorem 6, so $G$ has a decomposition complexity growth function $t \sim s$.
\end{proof}

\begin{corollary}
If
$1 \rightarrow K \rightarrow G \rightarrow H \rightarrow 1$
is a short exact sequence of groups such that $K$ is a countable group with straight finite decomposition complexity, $G$ is finitely generated, and $H$ has subexponential (constant, polynomial) decomposition growth, then $G$ has subexponential (constant, polynomial) decomposition growth.
\end{corollary}

\section{An example}
\label{ex}
Subexponential decomposition growth is weaker than both finite decomposition complexity and subexponential asymptotic dimension growth, yet it still implies property A. We conclude with an example of a group with infinite asymptotic dimension and subexponential decomposition growth, but whose finite decomposition complexity status and asymptotic dimension growth are both unknown.

We recall the wreath product of groups. Let $G,H$ be finitely supported groups, and we  define $G \wr H$. Let $F(H;G)$ be the set of finitely supported functions $f:H \to G$. This set can be identified with $\bigoplus_{h \in H} G$. Then $H$ acts on $F(H;G)$ by translation: for every $f \in F(H;G)$ and $h, h' \in H$, we let
$$hf(h') = f(h^{-1}h').$$
With this action we define the semi-direct product $G \wr H = F(H;G) \rtimes H$, called the wreath product of $G$ and $H$. Now we are ready to construct our example.

Let $F_2$ be the free group on two generators, and let $G$ denote Grigorchuk's group, the group of intermediate (subexponential) volume growth introduced in \cite{GR1}. The techniques of \cite{D1} give an upper bound on the dimension growth of the wreath product $\mathbb{Z} \wr H$ of growth type the volume growth of $H$. However, $F_2$ has exponential volume growth, and we conjecture that the dimension growth of $\mathbb{Z} \wr F_2$ is in fact exponential. Also, it is unknown whether $G$ has finite decomposition complexity (\cite{GTY2}, Question 5.1.3). Hence it is unknown whether $(\mathbb{Z} \wr F_2) \times G$ has either finite decomposition complexity or subexponential dimension growth.
\begin{theorem}
$(\mathbb{Z} \wr F_2) \times G$ has subexponential decomposition growth.
\end{theorem}
\begin{proof}
$\mathbb{Z}$ and $F_2$ both have finite asymptotic dimension, so they have finite decomposition complexity, and finite decomposition complexity is preserved by direct unions and extensions (see \cite{GTY1}). Hence $\mathbb{Z} \wr F_2$ has finite decomposition complexity, and by Proposition 1 it has subexponential decomposition growth. Also, the volume growth of a group is an upper bound on its dimension growth (\cite{DS1}, Lemma 2.3), so by Proposition 2 $G$ has subexponential dimension growth. Finally, subexponential decomposition growth is preserved by products by Theorem 4.1, so $(\mathbb{Z} \wr F_2) \times G$ has subexponential decomposition growth.
\end{proof}


\begin{thebibliography}{}
\bibitem{ANWZ1}
G. Arzhantseva, G. Niblo, N. Wright, J. Zhang, A characterization for asymptotic dimension growth, Algebr. Geom. Topol. 18 no. 1, 493-524 (2018).
\bibitem{BBF1}
M. Bestvina, K. Bromberg, K. Fujiwara, The asymptotic dimension of mapping class
groups is finite, preprint, arXiv:1006.1939.
\bibitem{BD1}
G. Bell, A. Dranishnikov, A Hurewicz-type theorem for asymptotic dimension and applications to geometric group theory, Trans. Amer. Math. Soc. 358, 4749-4764 (2006).
\bibitem{D1}
A. Dranishnikov, Groups with a polynomial dimension growth, Geom. Dedicata, 119, 1-15 (2006)
\bibitem{D2}
A. Dranishnikov, On asymptotic dimension of amalgamated products and right-angled Coxeter groups, Algebr. Geom. Topol. 8, 1281-1293 (2008).
\bibitem{DS1}
A. Dranishnikov, M. Sapir, On the dimension growth of groups, J. Algebra, 347 (1), 23-29 (2011)
\bibitem{DZ1}
 A. Dranishnikov, M. Zarichnyi, Asymptotic dimension, decomposition complexity, and Haver’s property C, Topol. Appl., 169, 99–107 (2014)
\bibitem{GR1}
 R. I. Grigorchuk, On Burnside's problem on periodic groups, Funktsional. Anal. i Prilozhen., 14 No. 1, 53-54 (1980).
\bibitem{GRO}
M. Gromov, Asymptotic invariants of infinite groups, Geometric group theory, Vol. 2, 1-295. Cambridge Univ. Press, Sussex (1991).
\bibitem{GTY1}
 E. Guentner, R. Tessera, G. Yu, A notion of geometric complexity and its applications to topological rigidity, Invent. Math., 189 no. 2 (2012)
 \bibitem{GTY2}
E. Guentner, R. Tessera, G. Yu, Discrete groups with finite decomposition complexity, Groups Geom.
Dyn. 7, no. 2, 377–402 (2011).
\bibitem{O1}
N. Ozawa, Metric spaces with subexponential asymptotic dimension growth, Int. J. Algebra Comput., 22 (2) (2012)
\bibitem{W1}
R. Willett, Some notes on Property A, Limits of graphs in group theory and computer science, 191–281 (2009)
\bibitem{WR1}
N. Wright, Finite asymptotic dimension for CAT(0) cube complexes, Geom. Topol. 16,
527–554 (2012).
\bibitem{Y1}
G. Yu, The Novikov conjecture for groups with finite asymptotic dimension, Ann.
of Math., (2) 147, 325-355 (1998).
\bibitem{Y2}
G. Yu, The coarse Baum-Connes conjecture for spaces which admit a uniform embedding into Hilbert space, Invent. Math., 139, 201-240 (2000).
\end{thebibliography}
\end{document}